\documentclass[a4paper]{article}

\usepackage{latexsym} 
\usepackage{amsmath}
\usepackage{epsf}
\usepackage[latin1]{inputenc}
\usepackage{graphics}
\usepackage{enumerate}
\usepackage{calc}
\usepackage{color}
\usepackage{amssymb}
\usepackage[active]{srcltx}
\usepackage{pstricks}
\usepackage{epsfig}

\usepackage{bbm}
\usepackage{amsthm}
\usepackage{amscd}
\usepackage{esint}
\usepackage{authblk}

\newcommand{\solutionOperatorStateEquation}{G}
\newcommand{\basisfunction}{\xi}
\newcommand{\energyStateEquation}{E}
\newcommand{\Length}{\text{Len}}
\newcommand{\Interface}{\text{Per}}

\newcommand{\F}{\mathcal{F}}

\newcommand{\p}{\phi}
\newcommand{\la}{\lambda}

\def\d{\partial}

\newcommand{\h}[1]{\widehat{#1}}
\def\e{\varepsilon}

\def\weaks{\stackrel{*}{\rightharpoonup}}

\def\weak{\rightharpoonup}
\newcommand{\R}{\mathbb{R}}

\def\FF{\mathcal{F}}
\renewcommand{\k}{\kappa}

\def\W{\mathcal{W}}

\renewcommand{\d}{\,\mathrm{d}}

\def\XXint#1#2#3{{\setbox0=\hbox{$#1{#2#3}{\int}$}
     \vcenter{\hbox{$#2#3$}}\kern-.5\wd0}}

\newcounter{bei}

\newcommand{\zwo}[2]{\begin{pmatrix} {#1}\\{#2} \end{pmatrix}}

\renewcommand{\t}{\widetilde}

\newtheorem{theorem}{Theorem}[section]
\newtheorem{lemma}[theorem]{Lemma}
\newtheorem{proposition}[theorem]{Proposition}
\newtheorem{corollary}[theorem]{Corollary}

\usepackage{pdfpages} 
\usepackage{pgfplots}
\pgfplotsset{compat=newest}

\begin{document}
\title{Material Optimization for Nonlinearly Elastic Planar Beams}
\author{Peter Hornung\footnote{FB Mathematik, TU Dresden, 01062 Dresden (Germany)}
, Martin Rumpf and Stefan Simon\footnote{Institut für Numerische Simulation, Universit\"at Bonn, 53115 Bonn (Germany)}}

%
%
%
\date{}
\maketitle

\begin{abstract}
We consider the problem of an optimal distribution of soft and hard material for nonlinearly elastic planar beams.
We prove that under gravitational force the optimal distribution involves no microstructure and is ordered,
and we provide numerical simulations confirming and extending this observation.
\end{abstract}
\begin{center}
\end{center}

\section{Introduction}\label{Sec0}

In this article we study shape optimization for nonlinearly elastic planar beams.
We consider the nonlinear bending theory for elastic plates derived
in \cite{FJM}. It assigns to a deformation $u : S\to\R^3$ of a given reference configuration
$S\subset\R^2$ the elastic energy (augmented by a potential energy term)
$$
\int_S |\Pi|^2 + \int_S f\cdot u.
$$
Here $\Pi$ is the second fundamental form of the immersion $u$ and $f : S\to\R^3$
is an external force. The key constraint on the deformation $u$ is that it must be 
an isometric immersion. Assuming that $S$ is a rectangle $(0,1)\times (-1, 1)$ and prescribing
clamped boundary conditions with $u(0,x_2) = (0,x_2,0)$ and $\partial_{x_1} u(0,x_2) = (1,0,0)$ 
for $x_2 \in (-1, 1)$, and with $f = (0,0,-1)$
modelling the gravitation, it is reasonable to assume that (the energy minimizing) $u$ will be of
the form $u(x) = (\gamma_1(x_1),x_2,\gamma_2(x_1)$, for some curve $\gamma : (0, 1)\to\R^2$.
Such an essentially one-dimensional deformation is equivalent to a planar beam.
\\
Changing labels and allowing for the use of two different materials (soft and hard), we are therefore
led to consider the following variational problem: For $\gamma\in W^{2,2}((0,1), \R^2)$
with $|\gamma'| = 1$ (this is the one-dimensional equivalent of the isometry constraint) and
with $\gamma(0) = 0$ and $\gamma'(0) = (1,0)$, consider the energy functional
\begin{equation}
\label{wgc}
\W(\gamma, \chi) = \frac{1}{2}\int_0^1 A(t)\kappa^2(t)\ \d t - \int_0^1 f(t)\cdot\gamma(t)\ \d t.
\end{equation}
Here, $\kappa$ denotes the curvature of $\gamma$ and
$
A(t) = (1 - \chi(t))a + \chi(t) b,
$
where $0 < a < b$ model the two material parameters 
and $\chi : (0, 1)\to\{0, 1\}$ describes the distribution of these two materials.
Thus, $\chi$ and $1-\chi$ are the  characteristic functions of  the hard phase and the soft phase, respectively.
\\
The compliance of a given material distribution $\chi$ is  
$\chi \mapsto \int_0^1 f\cdot\gamma + c_l\cdot\int_0^1 \chi\,.$
Here $c_l$ is a positive parameter, so that the second term penalizes the use of the harder phase $b$.
We seek to find the optimal design $\chi$, which is the one minimizing the  cost functional.
\begin{equation}
\label{intro-compli-2}
\chi\mapsto\int_0^1 f\cdot\gamma_{\chi} + c_l \int_0^1 \chi
\end{equation}
among all $\chi : (0, 1)\to \{0, 1\}$,
where $\gamma_{\chi}$ is the (as we will see) unique minimizer of $\gamma \mapsto \W(\gamma, \chi)$
for given $\chi$.
\\
A typical question is whether this minimum is attained, i.e., whether the optimal design
is `classical' in the sense that no microstructure occurs. Ideally, one would then
like to obtain more precise information about the optimal design. Our main analytical result
answers these questions (see Thm \ref{thm:optimalDesignClassical}):
\medskip

{\it The optimal design is classical. More precisely, there exists $t^\ast\in (0, 1)$
such that $A = b$ on $(0, t^\ast)$ and $A = a$ on $(t^\ast, 1)$.}
\medskip

Our numerical simulations confirm this result. We also go further
and consider numerically more general clamped boundary conditions allowing $\gamma'(0) \neq e_1$.

\section{Setting}

Throughout this article, $I$ denotes the interval $(0, 1)$.
As stated above, the isometry constraint imposed upon the deformation $\gamma : I\to\R^2$
of the reference configuration $I$ is $|\gamma'| = 1$.
For such $\gamma$, its curvature is $\k = \gamma''\cdot n$, where $n = (\gamma')^{\perp}$
is the normal. We are interested in deformations which are clamped at the left edge, i.e.,
$\gamma(0) = 0$ and $\gamma'(0) = e_1$; here and in what follows, $(e_1, e_2)$ denotes
the standard basis in $\R^2$.
\\
For a given function $A\in L^{\infty}(I; [a, b])$ and given
external force $f : I\to\R^2$, the total energy (elastic plus potential energy)
stored in the deformed configuration $\gamma : I\to\R^2$ is given by \eqref{wgc}.
The ansatz for $A$ is $A = \chi b + (1 - \chi)a$ 
for $\chi : I\to \{0, 1\}$. 
However, we will encounter more general $A$ as well.

We introduce the phase $K$ by setting $K(t) = \int_0^t\kappa(s) \d s$ and identify $\mathbb{C}$ and $\R^2$, so that
$$\gamma(t) = \int_0^t e^{iK(s)} \d s\,,$$ 
because $\gamma(0) = 0$. 
It is convenient to introduce $F = \int_t^1 f(s) \d s$ (so that $F' = -f$ and $F(1) = 0$).
Integrating by parts we have $\int_I f(t)\cdot\gamma(t) \d t = \int_I F(t)\cdot e^{iK}(t) \d t$, 
so \eqref{wgc} equals
\begin{equation}
\label{en-4}
\FF(K) := \frac{1}{2}\int_I A(t)(K'(t))^2\ \d t - \int_I F(t)\cdot e^{iK(t)}\ \d t.
\end{equation} 
Note that in terms of $K$, the clamped boundary condition is equivalent to $K(0) = 0$;
the condition $\gamma(0) = 0$ is automatically taken into account by our definition of $\gamma$.
Other boundary conditions can also be included into our scheme, e.g. following
the ideas described in \cite{H-PRE}.
In view of the boundary conditions, the natural space on which the functional $\FF$
given by \eqref{en-4} is defined, is the space
$$
X = \{K\in W^{1,2}(I) : K(0) = 0\}.
$$
Using the direct method in the calculus of variations one easily verifies that there exists a minimizer of $\FF$ on $X$.

\section{The state equation}

Performing variations within the space $X$, we see that critical points 
satisfy
\begin{equation}\label{en-3b} 
\int_I AK'(t)\ \p'(t) \d t - \int_I F(t)\cdot i\,e^{iK(t)}\ \p(t) \d t = 0\mbox{ for all }\p\in C_0^{\infty}((0, \infty)).
\end{equation}
By density, this is equivalent to the assertion that
$K\in X$ satisfies
the equilibrium equation
\begin{equation}
\label{en-3c}
(AK')' = - i\,e^{iK}\cdot F\mbox{ in }X',
\end{equation}
where $X'$ denotes the topological dual of $X$. 
This and formula \eqref{en-3b} are weak formulations of
$(AK')' = - i\,e^{iK}\cdot F$ subject to the boundary conditions
$K(0) = 0$ and $K'(1) = 0$.
The condition $K'(1) = 0$ arises as the natural boundary condition. 
More precisely,
we have the following lemma.

\begin{lemma}\label{lem-el} 
Let $A\in L^{\infty}(I; [a, b])$,  $F\in W^{1,2}(I)$, and assume that $K\in W^{1,2}(I)$ is a 
weak solution of 
\begin{equation}
\label{en-3} 
(AK')' = -F\cdot ie^{iK}\mbox{ in }I.
\end{equation}
Then $k := AK'$ is $C^1(\overline I)$, $K$ is Lipschitz,
and \eqref{en-3b} is equivalent to 
\begin{equation}
\label{en-8}
\begin{split}
k' &= -ie^{iK}\cdot F\mbox{ almost everywhere on }I \mbox{ and } k(1) = 0.
\end{split}
\end{equation}
Moreover, for any $t_0\in [0, 1]$ there exists at most one weak solution $K$ of \eqref{en-3}
with prescribed values of $K(t_0)$ and of $k(t_0)$.
\end{lemma}
\begin{proof}
The right-hand side of Equation \eqref{en-3} is continuous up to the boundary by the hypotheses, 
so clearly $k$ is $C^1$
up to the boundary. In particular, $k$ is bounded, so $K' = A^{-1}k$ implies that
$K$ is Lipschitz.
\\
To prove the asserted equivalence simply note that
for $\p\in C^{\infty}_0((0, \infty))$ the left-hand side of \eqref{en-3b}
equals
$$
\int_I k(t)\p'(t) \d t - \int_I F(t)\cdot ie^{iK(t)}\p(t) \d t = \int_I\left(- F(t)\cdot ie^{iK(t)} - k'(t) \right)\p(t) \d t \,.
$$
This is zero for all $\p\in C^{\infty}_0((0, \infty))$ if and only if 
\eqref{en-8} is satisfied.
\\
To prove uniqueness, we combine \eqref{en-3} with the definition of $k$:
\begin{equation}
\label{en-3a}
\zwo{K}{k}' = \zwo{A^{-1}k}{-F\cdot ie^{iK}},
\end{equation} 
which is an ODE system of the form $u' = G(t, u)$ with uniformly Lipschitz $G$.
Hence its solutions are uniquely determined by their value at a single point.
\end{proof}

\subsection{The state equation for particular forces}

In this section and the next, $A\in L^{\infty}(I; [a, b])$.
Recall that $n = ie^{iK}$ and that $k = AK'$.

\begin{lemma}\label{bigle}
Let $f_0 \in \R^2$ be fixed with $|f_0|=1$ and assume that $f\parallel f_0$ and $f\cdot f_0 > 0$
almost everywhere on $I$. Let $K\in W^{1,2}(I)$ solve 
\begin{equation}
\label{el-int}
(AK')' = - F\cdot ie^{iK}\mbox{ in the sense of distributions on }I;
\end{equation}
so we make no assumptions on the boundary data of $K$.
Then the following are true:
\begin{enumerate}[(i)]
\item \label{bahn-ii}
If $F\cdot ie^{iK} = 0$ on a set of positive length, then $K$ is constant on $I$,
with $e^{iK}\parallel f_0$.
\item \label{bahn-i}
If there exists $c\in\R$ such that $\{t\in I : k(t) = c\}$ has positive length,
then $K$ is constant on $I$, with $e^{iK}\parallel f_0$. In particular, $c$ must be $0$.

\end{enumerate}
\end{lemma}
\begin{proof}
We claim that 
\begin{equation}
\label{le-sw2}
k = 0\mbox{ almost everywhere on }\{F\cdot n = 0\}.
\end{equation}
In fact, almost everywhere on this set we have
$$
0 = (F\cdot n)' = - f\cdot n - \kappa F\cdot\gamma' = -\kappa F\cdot\gamma',
$$
where we have used that $f\cdot n = 0$ because $f \parallel F$.
However, by hypothesis $0 \neq F = (F\cdot\gamma')\gamma'$ because $F\cdot n = 0$.
Hence we conclude that indeed $\kappa = 0$ and thus $k=0$.
\\
To prove \eqref{bahn-ii},
note that $k = 0$
almost everywhere on $\{F\cdot ie^{iK} = 0\}$, by \eqref{le-sw2}. In particular
there exists a point $t_0\in (0, 1)$
with $k(t_0) = 0$ and $F(t_0)\cdot ie^{iK(t_0)} = 0$, too. But then
clearly $K \equiv K(t_0)$ and $k \equiv 0$ is a solution of \eqref{en-3},
because the direction of $F$ is constant. By Lemma \ref{lem-el} this is the only solution.
\\
Finally, since $F(t_0)\perp ie^{iK(t_0)}$, we know that $e^{iK}\parallel f_0$.
\\
To prove \eqref{bahn-i}, 
let $c\in\R$ be such that the set $\{k = c\}$ has positive length.
As $k' = 0$ almost everywhere on $\{k = c\}$, by \eqref{el-int} we then have $F\cdot n = 0$ 
on a set of positive length. Hence part \eqref{bahn-ii} implies
that $K$ is constant with $e^{iK}$ parallel to $f_0$.
\end{proof}

\begin{corollary}\label{corbigle}
Under the hypotheses of Lemma \ref{bigle} and assuming, in addition,
that $f_0$ is not parallel to $e^{iK(0)}$, we have $k'\neq 0$ almost everywhere.
In particular, the set $\{k = c\}$ has length zero for every $c\in\R$.
\end{corollary}
\begin{proof}
By \eqref{el-int} we have $F\cdot ie^{iK} = 0$ almost everywhere on $\{k' = 0\}$.
So if $k' = 0$ on a set of positive length, then Lemma \ref{bigle} \eqref{bahn-ii} would 
imply that $K$ is a constant satisfying $e^{iK}\parallel f_0$, contradicting
the relation between $f_0$ and $K(0)$.
\end{proof}

\subsection{Properties of minimisers}\label{Stateq}

The proof of the next proposition is based on energy comparison arguments. 

\begin{proposition}\label{kbd} 
Let $\beta\in (-\pi, 0)$ and let $f\parallel e^{i\beta}$ with $f\cdot e^{i\beta} > 0$
almost everywhere. Let $A\in L^{\infty}(I; [a, b])$ and let $K$ be an absolute minimiser of $\FF$
among all $K\in X$. Then $K' \leq 0$ almost everywhere on $I$, and on $[0, 1)$ the function
$K$ takes values in $(\beta, 0]$.
\end{proposition}
\begin{proof}
Let $t_1$ be the minimum over all $t$ such that
$
- \int_0^t |K'| = \beta;
$
if no such $t$ exists then set $t_1 = 1$. Define
$$
\t K = 
\begin{cases}
- \int_0^t |K'| &\mbox{ for }t\in [0, t_1]
\\
\beta &\mbox{ for }t > t_1.
\end{cases}
$$
Since
\begin{eqnarray*}
\F(\t K) - \F(K) &=&
- \frac{1}{2}\int_{t_1}^1 A(K'(t))^2 \d t+ \int_I F(t)\cdot \left(e^{iK(t)} - e^{i\t K(t)} \right) \d t \\
&\leq& \int_I F(t)\cdot \left(e^{iK(t)} - e^{i\t K(t)}\right) \d t,
\end{eqnarray*}
the hypotheses on $f$ readily imply that $\F(\t K)$ is strictly less than $\F(K)$
if $K' > 0$ on a set of positive length. Since $K$ is an absolute minimiser,
we therefore conclude that $K'\leq 0$ on $I$.
\\
We claim that $K > \beta$ on $[0, 1)$. In fact,
otherwise, by continuity and since $K' \leq 0$, there would be a point $t_1\in [0,1)$ at which
$e^{iK(t_1)} = e^{i\beta}$. Then by energy minimality we would have $K = K(t_1)$
on $[t_1, 1]$.
From this we obtain that $k = AK' = 0$ on $[t_1,1]$.
Hence, Lemma \ref{bigle} \eqref{bahn-i}
would imply that $K = K(t_1)$ everywhere on $I$, contradicting the boundary condition $K(0) = 0$.
\end{proof}

For simplicity, in what follows we assume that $f = e^{-i\pi/2} = -e_2$.
Motivated by Proposition \ref{kbd} we introduce the convex subset
$$
\t X = \left\{K\in X : K\in (-\frac{\pi}{2}, 0]\mbox{ on }[0, 1) \right\}.
$$
Observe that $K(1) = -\frac{\pi}{2}$ is not excluded, so this is
not an open condition.

\begin{proposition}\label{unique}
Let $A\in L^{\infty}(I; [a, b])$ and let $f = -e_2$. Then 
there exists at most one global minimizer of $\FF$ within $X$.
\end{proposition}
\begin{proof}
The claim in fact is a direct consequence of Proposition \ref{kbd} and convexity
of the energy density. We include the details for the reader's convenience.
The energy density
\begin{equation}
\label{unique-0} 
W(t, z, p) = \frac{1}{2}A(t)p^2 + (1 - t)\sin z.
\end{equation} 
satisfies
\begin{equation}
\label{unique-1}
W\left( t, \frac{p + \t p}{2}, \frac{z + \t z}{2} \right) 
\leq \frac{1}{2}W\left( t, p, z\right) + \frac{1}{2}W\left( t, \t p,\t z\right)
\end{equation} 
whenever $t\in I$ and $p, \t p\in \R$ and $z, \t z\in [-\frac{\pi}{2}, 0]$, and the inequality
in \eqref{unique-1} is strict unless $(p, z) = (\t p, \t z)$. These facts follow from the convexity
of the sine function on the intervals in question.
\\
If $K$, $\t K\in X$ are minimizers of $\FF$ within $X$, then by Proposition \ref{kbd}
we have $K$, $\t K\in \t X$.
Set $\h K = \frac{1}{2}\left( K + \t K \right)$. By \eqref{unique-1} we have
\begin{align*}
\min_X \FF &\leq \FF(\h K) = \int_I W\left(t, \frac{K(t) + \t K(t)}{2}, \frac{K'(t) + \t K'(t)}{2}\right) \d t
\\
&\leq \int_I \left( \frac{1}{2}W(t, K(t), K'(t)) + \frac{1}{2}W(t, \t K(t),\t K'(t)) \right) \d t
\leq \min_X \FF.
\end{align*}
Hence we have equality throughout. Again by \eqref{unique-1} this implies
\begin{equation}
\label{unique-2} 
W\left(\cdot, \frac{K + \t K}{2}, \frac{K' + \t K'}{2}\right)
=  \frac{1}{2}W(\cdot, K, K') + \frac{1}{2}W(\cdot, \t K,\t K')\mbox{ a.e. on }I,
\end{equation} 
so $K = \t K$ almost everywhere.
\end{proof}

\begin{proposition}
\label{eindeu}
If $f = -e_2$ and $K\in\t X$ satisfies $(AK')' = (1 - t)\cos K$ in $X'$ then $K$ is the (unique)
minimiser of $\FF$ on $X$.
\end{proposition}
\begin{proof}
Let $K\in\t X$ be as in the hypothesis.
By Proposition \ref{kbd} it is enough to show that $K$ is minimizing within $\t X$.
By convexity of $(z, p)\mapsto W(t, z, p)$ we have
$$
W(t, \t z, \t p) \geq W(t, z, p) + (\d_z W)(t, z, p) (\t z - z) +
(\d_p W)(t, z, p)(\t p - p)
$$
whenever $p$, $\t p\in\R$ and $z$, $\t z\in [-\frac{\pi}{2}, 0]$.
If $\t K\in \t X$, then we may insert
$(z, p) = (K, K')$ and $(\t z, \t p) = (\t K, \t K')$. Then we integrate
and use the equation satisfied by $K$ to find that indeed $\FF(\t K)\geq \FF(K)$.
\end{proof}

\section{Relaxation by the homogenization method}

For $\theta\in [0, 1]$ define
\begin{equation}
\label{Atheta}
A\left( \theta \right) = \left( \frac{1-\theta}{a} + \frac{\theta}{b}\right)^{-1}.
\end{equation}
If $\theta = \chi$ only takes values in $\{0, 1\}$, then
$$
A(\chi) = (1 - \chi)a + \chi b.
$$
The coefficient \eqref{Atheta} will arise naturally for the usual reason:
if $\chi_n\in L^{\infty}(I; \{0, 1\})$ converge weakly-$*$ in $L^{\infty}(I)$ 
to $\theta$, then
\begin{equation}
\label{rema-1}
\left( (1 - \chi_n)a + \chi_n b \right)^{-1} = \left( A(\chi_n) \right)^{-1} =  (1 - \chi_n)\frac1a + \chi_n \frac1b  \weaks \left( A(\theta) \right)^{-1}
\end{equation}
in $L^{\infty}(I)$.
We define the compliance $J : X\times L^{\infty}(I; [0, 1])\to\R$ as follows:
$$
J\left( K, \theta \right) = \int_I F(t)\cdot e^{iK(t)} \d t + c_l\int_I\theta(t) \d t.
$$
The constant $c_l$ is strictly positive, so the second term penalises the use of the hard material.
\\
The optimal design $\chi$ should minimise $J(K, \chi)$, under the constraint
that $K$ be a solution to
\eqref{en-3c} with $A = A(\chi)$,
among all $\chi\in L^{\infty}(I; \{0, 1\})$. 
Following the work \cite{Allaire}  in the context of linearised elasticity, we 
by deriving the corresponding relaxed problem and obtain the  following result:
\begin{proposition}\label{pro1} 
Let $\theta_n\in L^{\infty}(I; [0, 1])$ and let $\theta\in L^{\infty}(I; [0, 1])$ be such that
$\theta_n\weaks\theta$ weakly-$*$ in $L^{\infty}(I)$ as $n\to\infty$. Let 
$K_n\in X$ be a solution of \eqref{en-3c} with $A = A(\theta_n)$.
Then, after passing to a subsequence, $K_n$ converge weakly in $W^{1,2}(I)$
to a solution $K\in X$ of \eqref{en-3c} with $A = A(\theta)$.
\end{proposition}

Proposition \ref{pro1} is a consequence of the fact that under its hypotheses we have
$\left( A(\theta_n) \right)^{-1}\weaks \left( A(\theta) \right)^{-1}$ 
weakly-$*$ in $L^{\infty}$, and of the
following lemma.

\begin{lemma}\label{le1} 
Let $A_n\in L^{\infty}((0, 1), [a, b])$ and let 
$K_n\in X$ be a solution of \eqref{en-3c} with $A = A_n$, 
and suppose that there is $B\in L^{\infty}(I)$ such that 
$A_n^{-1}\weaks B$ weakly-$*$ in $L^{\infty}(I)$. 
Then there exists $K\in X$ such that, after passing to subsequences,
$K_n\weak K$ in $W^{1,2}(I)$. Moreover, $K$ solves \eqref{en-3c} with $A = B^{-1}$.
\end{lemma}
\begin{proof}
The state equation \eqref{en-3c} implies an a priori estimate for $K_n$:
in fact, testing \eqref{en-3c} with $K_{n}$ we have
\begin{align*}
a\int_I (K_n')^2 \d t&\leq \int_I A_n (K_n')^2 \d t = -\int_I K_n\ ie^{iK_n}\cdot F \d t
\leq \|F\|_{L^2}\|K_n\|_{L^2}.
\end{align*}
Since $K_n(0) = 0$ we have $\|K_n\|_{L^2}\leq \|K_n'\|_{L^2}$, so the above estimate implies
\begin{equation}
\label{}
\|K_n\|_{L^2} \leq \frac1a \|F\|_{L^2}.
\end{equation}
But then using the above chain of estimates again,
$$
\|K_n'\|_{L^2} \leq \frac1a  \|F\|_{L^2}.
$$
Hence, after taking subsequences, there is $K\in X$ such that $K_n\weak K$ in $W^{1,2}(I)$.
\\
Since $(A_nK_n')(1) = 0$ by Lemma \ref{lem-el}, we can write \eqref{en-3c} as 
$$
K_n' = A_n^{-1}\ \int_t^1 F\cdot ie^{iK_n} \d t.
$$
Since $K_n\to K$ uniformly, we have
$$
\int_t^1 F\cdot ie^{iK_n} \d t \to \int_t^1 F\cdot ie^{iK} \d t
$$
uniformly on $I$. Since $A_n^{-1}\weaks B$ in $L^{\infty}$, we deduce that $K$ satisfies
$$
K' = B\ \int_t^1 F\cdot ie^{iK} \d t.
$$
This is equivalent to \eqref{en-3c} with $A = B^{-1}$.
\end{proof}

Proposition \ref{pro1} can be viewed as a homogenization result
for the equilibrium equation of the nonlinear
bending energy functional \eqref{wgc}. Related (general) homogenization results 
for nonlinearly elastic rods can be found
in \cite{Velcic}, where the homogenization process is carried out on a variational
level (not on the equilibrium equation). The starting point in \cite{Velcic}
is the genuinely three-dimensional nonlinear elasticity functional for a rod of finite positive
thickness, and the homogenization limit is combined with the zero thickness limit.
\\

Now suppose that $f  = - e_2$.
Proposition \ref{eindeu} 
shows that for every $\theta\in L^{\infty}(I; [0, 1])$
there exists a unique solution $K\in \t X$
of \eqref{en-3c} with $A = A(\theta)$. 
Abusing notation we will denote this solution $K\in\t X$ by $K(\theta)$.
We define $\h J : L^{\infty}(I; [0, 1])\to\R$ by
$$
\h J(\theta) = J\left( K({\theta}), \theta \right).
$$

\begin{proposition}
Let $f = -e_2$. Then the infimum
\begin{equation}
\label{relax-2}
\inf_{\theta\in L^{\infty}(I; [0, 1])}\h J(\theta)
\end{equation}
is attained and agrees with
\begin{equation}
\label{relax-1}
\inf_{\chi\in L^{\infty}(I; \{0, 1\})} \h J(\chi).
\end{equation}
\end{proposition}
\begin{proof}
In order to see that \eqref{relax-2} is attained, let
$\theta_n\in L^{\infty}(I; [0, 1])$ be such that $\h J(\theta_n)$
converges to \eqref{relax-2}. After taking subsequences (not relabelled), we may assume that
$\theta_n\weaks \theta$ in $L^{\infty}(I)$.
Hence by Proposition \ref{pro1} we know that taking another subsequence $K_n = K(\theta_n)$ converge weakly in $W^{1,2}(I)$
to a solution $K\in X$ of \eqref{en-3c} (with $A = A(\theta)$).
By convexity of $\t X$, we know that $K\in\t X$. Hence $K = K(\theta)$ by Proposition \ref{eindeu}.
Hence
$
\h J(\theta_n) \to \h J(\theta).
$
\\
In order to prove that \eqref{relax-1} does not exceed \eqref{relax-2} (the other estimate is trivial),
let $\theta$ minimise $\h J$ among all functions in $L^{\infty}(I; [0, 1])$.
Let $\chi_n\in L^{\infty}(I; \{0, 1\})$ be such that $\chi_n\weaks\theta$.
Then as before we see that
$K(\chi_n)$ subconverge to $K(\theta)$ weakly in $W^{1,2}$, and therefore $\h J(\chi_n)\to\h J(\theta)$. 
\end{proof}

\section{Optimal design}\label{sec:optimalDesign}
Now, the natural question is whether microstructure actually occurs, i.e., whether
the minimum in \eqref{relax-1} is attained or not. Throughout this section we continue to assume $f = -e_2$.
Following the abstract approach in \cite{HPUU} we introduce the operator
$
\solutionOperatorStateEquation : X\times L^{\infty} \to X' 
$
by setting
$$
\solutionOperatorStateEquation (K, \theta) = (A(\theta)K')' - (1 - t)\cos K.
$$
The optimal design is a function $\theta : I\to [0, 1]$ 
minimising $J(K, \theta)$ subject to the constraint that $K$ 
be the minimiser of the elastic energy $\FF$ with $A = A(\theta)$.
This constraint on $K$ is equivalent to the requirement that $K\in\t X$
be a solution of \eqref{stateq}, i.e., of 
$\solutionOperatorStateEquation (K, \theta) = 0$ in $X'$.
\\
In fact, by the results of Section \ref{Stateq} we know that for given $\theta\in L^{\infty}(I; [0, 1])$
there exists a unique solution $K \in \t X$ of the state equation
\begin{equation}
\label{stateq}
(A(\theta)K')' = (1 - t)\cos K,
\end{equation}
and this $K$ is the unique absolute minimiser of the functional $\FF$ with $A = A(\theta)$. As before,
we denote this minimiser by $K(\theta)$.

\begin{lemma}\label{KC1}
For $\e > 0$ small enough (depending on $a$ and $b$), the map $K : L^{\infty}(I; (-\e, 1 + \e))\to W^{1,2}(I)$ taking
$\theta$ into $K(\theta)$ is continuously Fr\'echet differentiable.
\end{lemma}
\begin{proof}
It is easy to verify that $\solutionOperatorStateEquation  : X\times L^{\infty}\to X'$ is continuously Fr\'echet differentiable.
Its partial Fr\'echet derivative $D_1 \solutionOperatorStateEquation (K, \theta)$ with respect to $K$ is the operator taking $\eta\in X$ into
\begin{equation}
\label{linle-1}
D_1 \solutionOperatorStateEquation (K, \theta)(\eta) = \left( A(\theta)\eta' \right)' + (1 - t)(\sin K)\eta.
\end{equation}
For $K\in\t X$ the linear operator $D_1 \solutionOperatorStateEquation (K, \theta) : X\to X'$ is easily seen to be bijective,
because $\sin K$ is nonpositive for $K\in\t X$.
Hence the claim follows from the implicit function theorem.
\end{proof}

Later on we will need the dual operator of $D_2 \solutionOperatorStateEquation (K, \theta)$. 
Clearly
$
D_1J(K, \theta) = F\cdot ie^{iK} = - (1 - t)\cos K
$
and
$
D_2J(K, \theta) = c_l.
$
The partial derivative $D_2 \solutionOperatorStateEquation (K, \theta) : L^{\infty}(I; (-\e, 1 + \e))\to X'$
is the linear map given by
$$
D_2\solutionOperatorStateEquation (K, \theta)(\eta) = \left( \dot A(\theta)\ \eta\ K' \right)'.
$$
Here 
\begin{equation}
\label{Adot}
\dot A(\theta) = \left( \frac{1}{a} - \frac{1}{b} \right)A^2(\theta),
\end{equation}
where $A(\theta)$ is as in \eqref{Atheta}.
Using this, we see that the dual operator $D_2 \solutionOperatorStateEquation (K, \theta)^* : X\to (L^{\infty})'$
to $D_2 \solutionOperatorStateEquation (K, \theta)$ is 
$$
D_2\solutionOperatorStateEquation (K, \theta)^* = \left( \left( \frac{1}{a} - \frac{1}{b} \right)A(\theta) k\right)',
$$
where $k=A(\theta) K'$ as in Lemma \ref{lem-el}.

\subsection{Equilibrium equation for the optimal design}

We will now derive the equilibrium equation satisfied by
$\h J$-minimising $\theta$.
Denoting the Fr\'echet derivative of $K$ 
with respect to $\theta$ by $DK$,
we compute (using Lemma \ref{KC1})
$$
D\h J(\theta)(\eta) = D_1J(K({\theta}), \theta)\left( DK(\theta)\eta \right) + 
D_2J(K(\theta), \theta)\eta
$$
for all $\eta\in L^{\infty}(I)$, that is,
$$
D\h J(\theta) = (DK(\theta))^* \left(D_1 J(K(\theta), \theta)\right) + D_2 J(K(\theta), \theta).
$$
In order to compute the first term on the right-hand side,
we differentiate the state equation $D_1 G(K(\theta),\theta) = 0$ with respect to $\theta$ and take adjoints to see that
$$
(DK({\theta}))^* = - D_2\solutionOperatorStateEquation \left(K(\theta), \theta)\right)^* \left((D_1 \solutionOperatorStateEquation (K(\theta), \theta))^{-1}\right)^*.
$$
Therefore, if $P\in X$ is the unique solution of
\begin{equation}
\label{adj} 
\left(D_1\solutionOperatorStateEquation (K({\theta}), \theta)\right)^*(P) = -D_1J(K({\theta}), \theta) \mbox{ in }X',
\end{equation} 
then
\begin{equation}
\label{J'}
D\h J(\theta) = D_2\solutionOperatorStateEquation (K({\theta}), \theta)^*P + D_2J(K({\theta}), \theta).
\end{equation}

By the computations above equation \eqref{J'} becomes
\begin{equation}
\label{J-1}
D\h J(\theta) = 
-\left( \frac{1}{a} - \frac{1}{b} \right) k p + c_l,
\end{equation}
where we introduced $p = A(\theta)P'$.
And the adjoint equation \eqref{adj} becomes 
\begin{equation}
\label{J-3} 
\left( A(\theta)P' \right)' = p' = (1 - t)\left( \cos K - P \sin K \right).
\end{equation}

The equilibrium equation satisfied by designs $\theta$ minimising $\h J$ asserts that
$$D\h J(\theta)(\eta) \geq 0$$
 for all $\eta\in L^{\infty}(I)$ satisfying
$\eta \geq 0$ almost everywhere on the set $\{\theta = 0\}$ and
$\eta \leq 0$ on $\{\theta = 1\}$.
\\
This leads to the following pointwise condition:
$$
\left( \frac{1}{a} - \frac{1}{b} \right) k p - c_l
\begin{cases}
\leq 0  &\mbox{ on }\{\theta = 0\}
\\
\geq 0  &\mbox{ on }\{\theta = 1\}
\\
= 0  &\mbox{ on }\{\theta\in (0, 1)\}.
\end{cases}
$$
Since $\left( \frac{1}{a} - \frac{1}{b} \right) > 0$, with
\begin{equation}
\label{def-lambda}
\la = \left( \frac{1}{a} - \frac{1}{b} \right)^{-1}c_l
\end{equation}
this can be written as follows:
\begin{equation}\label{el-theta}
k p
\begin{cases}
\leq \la &\mbox{ on }\{\theta = 0\}
\\
\geq \la &\mbox{ on }\{\theta = 1\}
\\
= \la  &\mbox{ on }\{\theta\in (0, 1)\}.
\end{cases}
\end{equation} 

\subsection{Properties of the adjoint variable}

In this section we continue to assume $f = -e_2$.
Since the right-hand side of \eqref{J-3} is continuous, we see that $p\in C^1([0, 1])$.
Recall from \eqref{stateq} that $K = K(\theta)$ satisfies
$$
(A(\theta)K')' = (1 - t)\cos K \mbox{ and }K(0) = 0,\ K'(1) = 0,
$$
and $K$ is decreasing and on $[0, 1)$ takes values in $(-\frac{\pi}{2}, 0]$.

In order to study the behaviour of $P$, we introduce $\rho : I\to\R$ by
$$
\rho(t) = \cot K(t) = \frac{\cos K(t)}{\sin K(t)},
$$
so clearly $\rho < 0$ on $(0, 1)$, and $\rho(t)\to -\infty$ as $t\downarrow 0$. 
Moreover,
\begin{equation}
\label{rho'} 
\rho' = -\frac{K'}{\sin^2 K}\mbox{ and }A\rho' = -\frac{k}{\sin^2 K}.
\end{equation} 
By Lemma \ref{lem-el} we see that $A\rho'$ is continuous,
positive and strictly decreasing on $(0, 1)$.
The relevance of $\rho$ is that $p'$
is a positive multiple of $Q := P - \rho$. In particular,
$p' = 0$ if and only if $Q = 0$, and the sign of $p'$ equals that of $Q$.
\\
We introduce 
$$
q := AQ' = p - A\rho' = p + \frac{k}{\sin^2 K}
$$
and we compute
\begin{equation}
\label{lej3-a}
q' = - (1 - t) Q \sin K + \left( \frac{k}{\sin^2 K} \right)'.
\end{equation}

\begin{lemma}
\label{lej1} 
We have $q(1) = p'(1) = p(1) = P(0) = 0$, as well as $p(0) < 0$ and $p'(0) = 1$.
\end{lemma}
\begin{proof}
We have $P(0) = 0$ because $P\in X$, and $p(1) = 0$ (hence $q(1) = 0$ since $k(1) = 0$)
because \eqref{J-3} is an equation in $X'$ involving natural boundary conditions. 
From \eqref{J-3} and since $K(0) = P(0) = 0$, 
we have
$$
p'(0) = \cos K(0) - P(0)\sin K(0) = 1.
$$
Also from \eqref{J-3}, we see $p'(1) = 0$.
Finally, the inequality $p(0) < 0$ follows easily from 
the boundary conditions $p(1) = 0$ and $P(0) = 0$ and the 
observation from \eqref{J-3} that $p'\geq 0$ on $\{P\geq 0\}$.
Indeed, assuming $p(0)>0$ we obtain a straighforward contradiction to $p(1)=0$ 
and assuming $p(0)=0$ we deduce that $p'\equiv 0$, which implies $K\equiv -\frac{\pi}{2}$ and this contradicts $K(0) =0$.
\end{proof}

\begin{lemma}
\label{lej4}
There is $t_0\in [0, 1]$ such that $Q > 0$ on $[0, t_0)$ and $Q\leq 0$ on $[t_0, 1]$.
\end{lemma}
\begin{proof}
As $Q(0) = +\infty$, it is enough to show that $Q' \leq 0$ almost everywhere
on $\{Q \geq 0\}$. As $A$ is positive, this is equivalent to the assertion that
$q \leq 0$ almost everywhere on $\{Q \geq 0\}$.
\\
By \eqref{lej3-a} we have 
\begin{equation}
\label{lej4-1}
q'\geq 0\mbox{ almost everywhere on }\{Q \geq 0\}
\end{equation}
because $k\cdot \sin^{-2} K$ is an increasing function. So if $t_0$ is such that
$Q(t_0)\geq 0$ and $q(t_0) > 0$, then $q$ is nondecreasing on $(t_0, 1)$, which 
can be seen as follows:
Since $q(t_0) > 0$, by continuity of $q$ the set
$$
\{t\in (t_0, 1) : q(t) > 0 \mbox{ on }(t_0, t) \}
$$
is nonempty. Denote by $t_1$ the supremum over this set.
Then $Q$ is increasing on $(t_0, t_1)$ because $Q' = q/A$.
Since $Q(t_0) \geq 0$, this implies that $Q \geq 0$ on $(t_0, t_1)$.
Hence $q$ is nondecreasing on $(t_0, t_1)$ by \eqref{lej4-1}.
Hence $q(t_1) > 0$, so by continuity necessarily $t_1 = 1$.
\\
Therefore, one obtains $q(1) > 0$, contradicting Lemma \ref{lej1}.
\end{proof}

\begin{proposition}
\label{lej5} 
There exists $t_0\in (0, 1]$ such that $p' > 0$ on $[0, t_0)$ and $p'\leq 0$ on $[t_0, 1]$.
Moreover, the following is true:
\begin{itemize}
\item If $t_0 = 1$ then $p < 0$ on $[0, 1)$.
\item If $t_0 < 1$ then there exists $t_1\in (0, t_0)$ such that
$p < 0$ on $[0, t_1)$ and $p > 0$ on $(t_1, 1)$. 
\end{itemize}
\end{proposition}
\begin{proof}
The first part follows from Lemma \ref{lej4} and our initial observation
that the sign of $p'$ is determined by that of $Q$.
\\
To prove the second part, first note that if $t_0 = 1$ then $p < 0$ on $[0, 1)$
because $p(1) = 0$ and $p$ is increasing.
\\
If $t_0 < 1$ then $p > 0$ on $(t_0, 1)$. In fact, since $p$ is nonincreasing on this interval and 
since $p(1) = 0$, if we had $p(t') = 0$ at some $t'\in (t_0, 1)$ then $p = 0$ 
on $(t', 1)$. By \eqref{J-3}
this would imply that $P = \cot K$ on this interval. And by $AP' = p = 0$ the function $\cot K$ 
and therefore $K$ and thus $k$ would be constant on $(t', 1)$. This would contradict
Corollary \ref{corbigle}.
\\
Since $p$ is strictly increasing on $(0, t_0)$ and $p(0) < 0$ by Lemma \ref{lej1},
and since $p(t_0) > 0$, by continuity there exists precisely one $t_1$ as in the statement.
\end{proof}

\begin{corollary}
\label{lej6}
There exists $t_2\in (0, 1]$ such that $kp > 0$ and $kp$ is strictly
decreasing on $(0, t_2)$ and $kp \leq 0$ on $[t_2, 1]$. In particular,
the set $\{kp = c\}$ has zero length for any $c > 0$.
\end{corollary}
\begin{proof}
Recall that $k$ is negative and strictly increasing.
Let $t_0$ and $t_1$ be as in the conclusion of Proposition \ref{lej5}.
If $t_0 = 1$ then $p$ is negative and strictly increasing $[0, 1)$,
so $kp$ is positive and strictly decreasing.
In this case, therefore, the claim is satisfied with $t_2 = 1$.
Finally, if $t_0\in (0, 1)$, then the claim is satisfied with $t_2 = t_1$.
\end{proof}

The above proof of Lemma \ref{lej4} is self-contained. For variety, we also include
a shorter proof based on the
following maximum principle:
\begin{lemma}\label{maximumprinciple} 
Let $q$, $m : [0, T]\to \R$ be measurable with $m > 0$ and $q\leq 0$ almost everywhere.
Let $u$ be locally absolutely continuous and
such that $mu'$ is locally absolutely continuous, and such that
$$
(mu')' + qu\leq 0\mbox{ almost everywhere on }(0, T)
$$
and $u(0)$, $u(T) \geq 0$.Then $u\geq 0$ on $(0, T)$.
\end{lemma}

A proof of Lemma \ref{maximumprinciple} can be found in \cite{Walter}.
In order to apply Lemma \ref{maximumprinciple},
we extend $A$, $P$ and $K$ (and thus $\rho$) evenly to $[0, 2]$ by setting
$$
B(1 + t) = B(1 - t) \mbox{ for }t\in (0, 1]
$$
and $B=A,\,P,\,K$. We introduce the operator $Lu = (Au')' + \left( (1-t)\sin K\right) u$.
So \eqref{lej3-a} becomes 
\begin{equation}
\label{mapr1}
LQ = \left( \frac{k}{\sin^2 K}\right)'\mbox{ on }(0, 2).
\end{equation}
As mentioned below \eqref{rho'}, 
the quantity $k\cdot\sin^{-2} K$ is strictly increasing on $(0, 1)$, hence
the right-hand side of \eqref{mapr1} is positive on $(0, 1)$.
As $A$ and $K$ are even about $1$, the function $k = AK'$ 
is odd about $1$, hence so is $k\cdot\sin^{-2} K$. Therefore the right-hand
side of \eqref{mapr1} is positive on $(0, 2)$.
\\

As $P(0) = 0$ and $\rho(0) = -\infty$, either $Q > 0$ on $[0, 1)$ or
there exists a smallest $t_0\in (0, 1)$ such that $Q(t_0) = 0$.
In the latter case, in view of \eqref{mapr1} and since both $P$ and
$\rho$ are even about $1$, the 
function $Q$ satisfies the boundary value problem
\begin{align*}
LQ > 0 &\mbox{ in }(t_0, 2 - t_0)
\\
Q = 0 &\mbox{ on }\partial (t_0, 2 - t_0).
\end{align*}
Hence Lemma \ref{maximumprinciple} implies that $Q \leq 0$ on $[t_0, 2 - t_0]$;
in particular on $[t_0, 1]$. And by definition $Q > 0$ on $[0, t_0)$.
Therefore we have recovered Lemma \ref{lej4}.

\subsection{The optimal design}

Since $c_l > 0$ and $0 < a < b$, we have $\la > 0$ by its definition in \eqref{def-lambda}. 
Combining \eqref{el-theta} with Corollary \ref{lej6}, we therefore 
obtain the following result (with $t^\ast < t_2$):
\begin{theorem}\label{thm:optimalDesignClassical}
The optimal design is classical and ordered.
More precisely, if $\theta\in L^{\infty}(I)$ is a critical point of $\h J$, then
there exists $t^\ast \in (0, 1)$ such that $A(\theta) = b$ almost everywhere on $(0, t^\ast)$
and $A(\theta) = a$ almost everywhere on $(t^\ast, 1)$.
\end{theorem}

In \cite{FonsecaFrancfort} the {\em worst} design for nonlinearly elastic
membranes was studied, with a nonlinear compliance consisting of the sum
of the compliance used here plus the elastic energy. (We refer to \cite{RumpfWirth}
for a discussion of various choices of compliances in the context of nonlinear
elasticity.)
\\
In our setting, too, this worst design problem is much
easier to handle than the optimal design. In fact, there is no need to consider
the adjoint variable $p$: instead of Corollary \ref{lej6}
one merely needs the observation that $k^2$ is not constant on any set
of positive length, which follows readily via the Leibniz rule from
the results in Section \ref{Stateq}. One can then show that the worst design is also
classical and ordered. As expected, the order is reversed with respect
to the optimal design: first the soft phase is used and then the hard phase.
We leave the details to the interested reader.

\section{Numerical discretization of the state equation}\label{sec:numericsStateEquation}

In this section we consider a force $f = - \delta \ e_2$ for $\delta \in \R$,
and we allow inhomogeneous clamped boundary conditions $K(0)=K_0$ ($\dot \gamma(0) = e^{iK_0}$). 
The corresponding curve is given by
\begin{align*}
 \gamma(t) = \int_0^t e^{i(K(s)+K_0)} \ d s \, ,
\end{align*}
where $K \in X = \{K\in W^{1,2}(0,1) : K(0) = 0\}$.
Then, the associated stored energy is given by
\begin{align*}
  \energyStateEquation(K) = \int_0^1 \frac{1}{2} A (K')^2 + \delta (1-t) \sin( K(t) + K_0 ) \ dt \, .
\end{align*}
We use Newton's method to find local minimizer of the stored energy.
It requires to compute the first and second derivatives of the stored energy:
\begin{align*}
 & D\energyStateEquation(K)(\phi) = \int_0^1 A K' \phi' + \delta (1-t) \cos( K(t) + K_0 ) \phi \ dt\,, \\
 & D^2\energyStateEquation(K)(\phi)(\psi) = \int_0^1 A \phi' \psi' - \delta (1-t) \sin( K(t) + K_0 ) \phi \psi \ dt \,,
\end{align*}
where $\phi, \psi \in X$.

\par\bigskip
For the numerical implementation we consider a piecewise affine and continuous Finite Elements. In explicit, we take into account an equidistant grid with $N$ nodes $x_n=\frac{n}{N-1}$ for $n=0,\ldots,N-1$ and associated $N-1$ cells $(x_{n-1},x_{n})$ for $n=1,\ldots, N-1$.
The corresponding grid width is given by $h = \frac{1}{N-1}$.
Then we approximate $K$ in the space $V_h$ of functions, which are continuous and piecewise affine on the above cells.
Here and in what follows, we identify finite element functions and the corresponding coordinate vectors in the hat basis.
We denote the nodal basis functions of $V_h$ by $\basisfunction_h^n$ for $n=0,\ldots,N-1$. 
For the numerical integration, we choose a Gaussian quadrature with $Q$ quadrature points per element, where we use $Q=5$ in the implementation
and obtain the approximation 
\begin{align}\label{eq:numericalQuadrature}
 \int_0^1 g(t) \ dt \approx \sum_{l \in I_C} \sum_{q=0}^{Q-1} w_q^l g(x_q^l)
\end{align}
with $w_q^l$ denoting the weight at the quadrature point $x_q^l$.
Applying this quadrature to the stored energy and its derivatives, we get a discrete stored energy 
$\energyStateEquation_h$ on $V_h$ and associated derivatives $D\energyStateEquation_h$, and $D^2\energyStateEquation_h$.

Testing the first derivative with the basis functions, we obtain a vector $R[K] := (R[K]_j)_{j=0,\ldots,N-1}$ with 
$R[K]_j = D\energyStateEquation_h(K_h)(\basisfunction_h^j)$.
Analoguesly, testing the second derivative, we are led to a matrix $M[K] = (M[K]_{ij})_{i,j=0,\ldots,N-1}$ with
$$M[K]_{ij} = D^2\energyStateEquation_h(K_h)(\basisfunction_h^j)(\basisfunction_h^i)\,.$$
Because of the clamped boundary conditions we modify the first row and column of $M[K]$ by setting $M[K]_{0,0} = 0$ and $M[K]_{0,j} = 0 = M[K]_{i,0}$ for $i,j = 1,\ldots,N-1$, and we set $R[K]_0 = 0$.
Finally, Newton's  method for minimization of the stored energy computes a sequence $(K_h^{i})_{i=1,\ldots}$ with 
\begin{align*}
 M[K_h^{i}] ( K_h^{i+1} - K_h^{i} ) = R[K_h^{i}] 
\end{align*}
for given initial data  $K_h^0$. 
To cope with the nonlinearity, we use a multilevel scheme, first solving the problem on a coarse grid, prolongate the obtained result onto a finer grid, and proceed iteratively. 
Here, we take into account a dyadic sequence $N=2^l + 1$ with $l=L_c,\ldots, L_f$, where we usually use $L_c=3$ and $L_f \in {9,10,11}$.

\par\bigskip

For a homogeneous material $A\equiv 1$ we experimentally observe essentially three types of stationary points (see Fig.~\ref{fig:StateEquation}).
First, there is of course a simple configuration where the curve is just turning downwards.
In fact, this appears to be and approximation of the global minimizer of the energy functional $\energyStateEquation$ discussed
in the first part of this article.
Secondly, we get a twisted curve, which can be interpreted physically
as turning the free end of the beam to the other side.
These two configurations are relatively stable under a change of material, i.e.,
taking some simple (resp. twisted) beam as initialization for a different material,
the computed discrete solution in our experiments always turned out to be a simple (resp. twisted) beam again. 
However, there is also a highly unstable configuration in between,
where the beam neither decides to fall to left side nor the right side.

\begin{figure}[!htbp]
\resizebox{1.0\linewidth}{!}{
\includegraphics{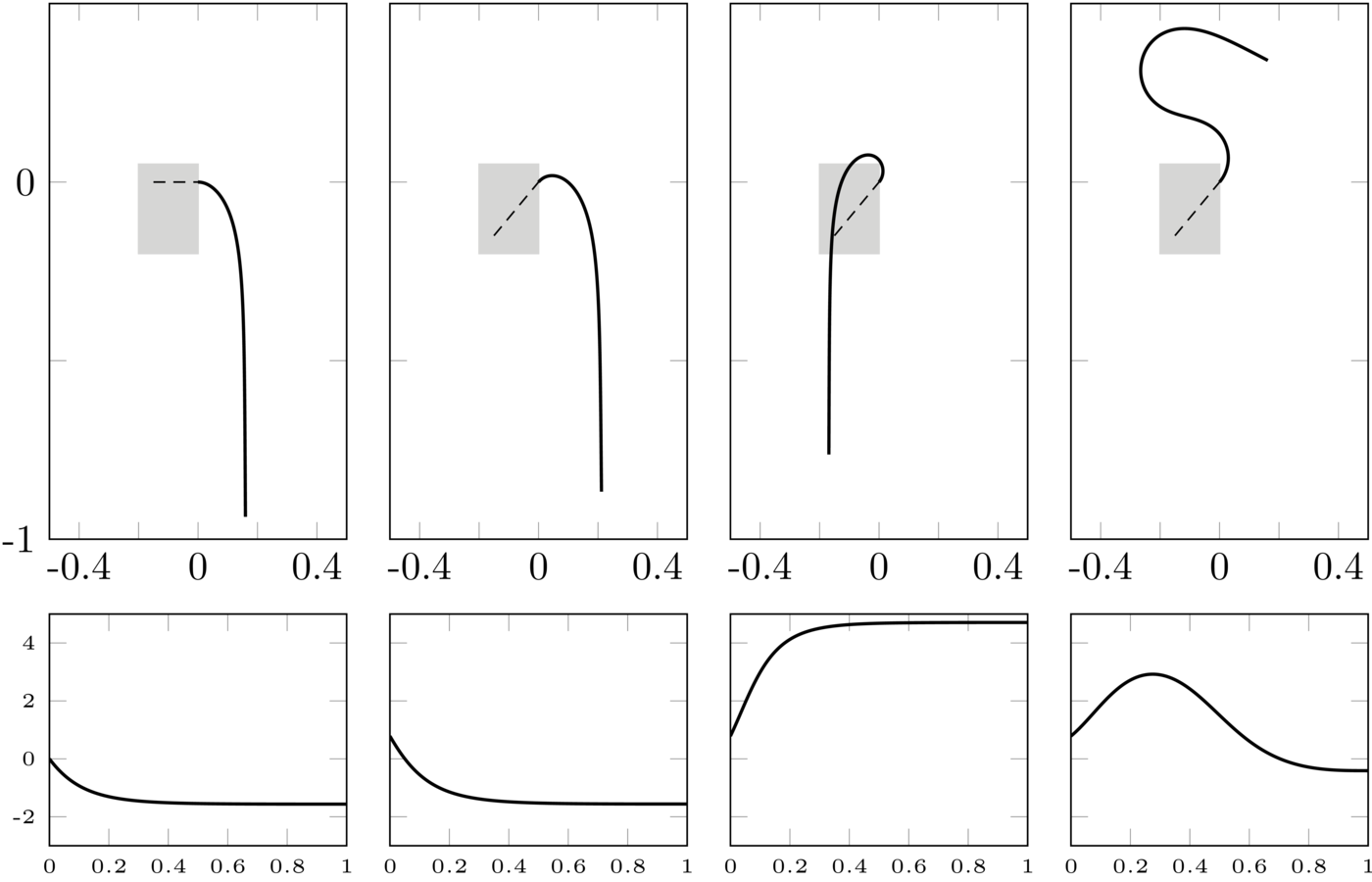}
}
\caption{Different solutions of the state equation (top row) with corresponding phase variable $K$ (bottom row) are shown (from left to right):
simple configurations with $K_0$ and $K_0 = \tfrac{\pi}{4}$, a twisted beam with $K_0 = \tfrac{\pi}{4}$, and an S-shaped configuration with $K_0 = \tfrac{\pi}{4}$.
Here, we have chosen $\delta = 100$, $A = 1$.}
\label{fig:StateEquation}
\end{figure}

\section{Computing optimal designs}\label{sec:numericsOptDesign}
Our numerical scheme to compute the optimal design is based on a phase field approach.
Following \cite{RumpfWirth} we take into account a phase field function $v:[0,1] \to \R$ with takes values either approximately $1$ for hard material 
with elasticity constant b and approximately $-1$ for soft material with elasticity constant $a$.
Thus, the material coefficient $A$ is assumed to be a function of $v$ and at each point $t \in [0,1]$ 
\begin{align*}
 A(v) = b \chi(v) + a (1-\chi(v) ) \, ,
\end{align*}
where we approximate the characteristic function $\chi$ by
\begin{align*}
 \chi(v)  = \frac{1}{4} (v + 1)^2 \, .
\end{align*}
To ensure the phase-field function to be smooth and essentially to take values $v \in \{ -1, 1 \}$, we use the  1D version of the perimeter functional proposed by 
Modica and Mortola \cite{MoMo77}
\begin{align*}
 \Interface^\epsilon(v) & = \frac{1}{2}\int_0^1 \epsilon \| v' \|^2 + \frac{1}{\epsilon} \frac{9}{16} (v^2-1)^2 \ d t 
\end{align*}
as regularizer, where $\epsilon$ describes the width of the diffuse interface.
Further, the definition of $\chi$ allows us to approximate the length covered by hard material by
\begin{align*}
  \Length(v) = \int_0^1 \chi(v) \ d t \ .
\end{align*}
Altogether, this allows us to define in analogy to Section~\ref{sec:optimalDesign} the (augmented) compliance functional as
\begin{align}
 J(K,v) = \int_0^1 - \delta (1-t) \sin(K(t) + K_0) \ d t + c_{l} \Length(v) + c_{p} \Interface^{\epsilon}(v) \, ,
\end{align}
with coeffcients $c_{l}, c_{p} > 0$.
Thus, the total cost functional in terms of a phase field function is given by
\begin{align} 
 \h J(v) = J(K(v),v) \, ,
\end{align}
where $K(v)$ is a solution to $DE(K)(\phi) = 0$ for all test functions $\phi \in X$ and $E$ takes into account the material coefficient $A(\theta)$.
The task is now to minimize $\h J$ over all phase fields $v$.
For this purpose we can apply the same abstract approach as in Section~\ref{sec:optimalDesign} and obtain as derivative
\begin{align}\label{eq:DerivativeJ}
 D\h J(v)(w) = D_v J (K(v),v)(w) + (D_v D_KE)^{*}(K(v),v) P 
\end{align}
where $P$ is the adjoint variable solving
\begin{align}\label{eq:dualProblemPhasefield}
 (D_K D_KE)^{*}(K(v),v) P = - D_K J (K(v),v) \, .
\end{align}
This requires the derivatives
\begin{align*}
 & D_v J(K,v)(w) = c_{l} \int_0^1 \frac{1}{2} (v+1) w \ d t + c_{p} \int_0^1 \epsilon v'  w' + \frac{9}{8 \epsilon} (v^2 - 1) v w \ d t \\
 & D_K J(K,v)(\phi) = \int_0^1 - \delta (1-t) \cos(K(t) + K_0) \phi \ d t \\
 & D_v D_KE(K,v)(\phi)(w) = \int_0^1 \frac{1}{2} (b-a) (v+1) w K' \phi' \ d t \\
 & D_K D_KE(K,v)(\phi)(\psi) = \int_0^1 A(v) \phi' \psi' - \delta (1-t) \sin(K(t) + K_0) \phi \psi \ d t \, .
\end{align*}
We choose $v_h$ in the finite element space $V_h$ defined in Section~\ref{sec:numericsStateEquation}. 
Let us emphasize that we have to impose the Dirichlet boundary condition for $P$, i.e. $P_h(0) = 0$.
Using the numerical quadrature in \eqref{eq:numericalQuadrature}, we obtain discrete operators $\h J_h$, $J_h$, $\Length_h$, $\Interface^\epsilon_h$, $D_KE_h$, and the corresponding derivatives.
With these functionals and operators at hand, we use the Quasi-Newton-Method (BFGS) to compute minimizers of $\h J_h$.

For the optimization of the phase variable  $K_h(v_h)$ for fixed $v_h$ we can proceed as in Section~\ref{sec:numericsStateEquation}.
Note that, in general, for a given phase field function $v_h$ the solution $K_h(v_h)$ is not necessarily unique,
since we have seen that different solutions of the state equation are possible.
In our numerical scheme, starting with some initial phase, our Newton-Method converges to a state
$K_h(v_h)$ which depends upon this initialization. 

Our numerical experiments reflect the result from Theorem~\ref{thm:optimalDesignClassical} (see Fig.~\ref{fig:optimalDesign}).
Furthermore, they suggest that a similar result remains true for
solutions of the state equation other than the absolute minimizer.
In fact, in our numerical simulations for clamped boundary conditions at $0$ 
the optimal design always gathers the hard material on the left in some interval $[0,t^\ast]$.
In Fig. ~\ref{fig:optimalDesign} we only depict one instance of many tests we performed with three different numerically computed 
local minimizers of the cost functional.

\begin{figure}[!htbp]
\resizebox{1.0\linewidth}{!}{
 \includegraphics{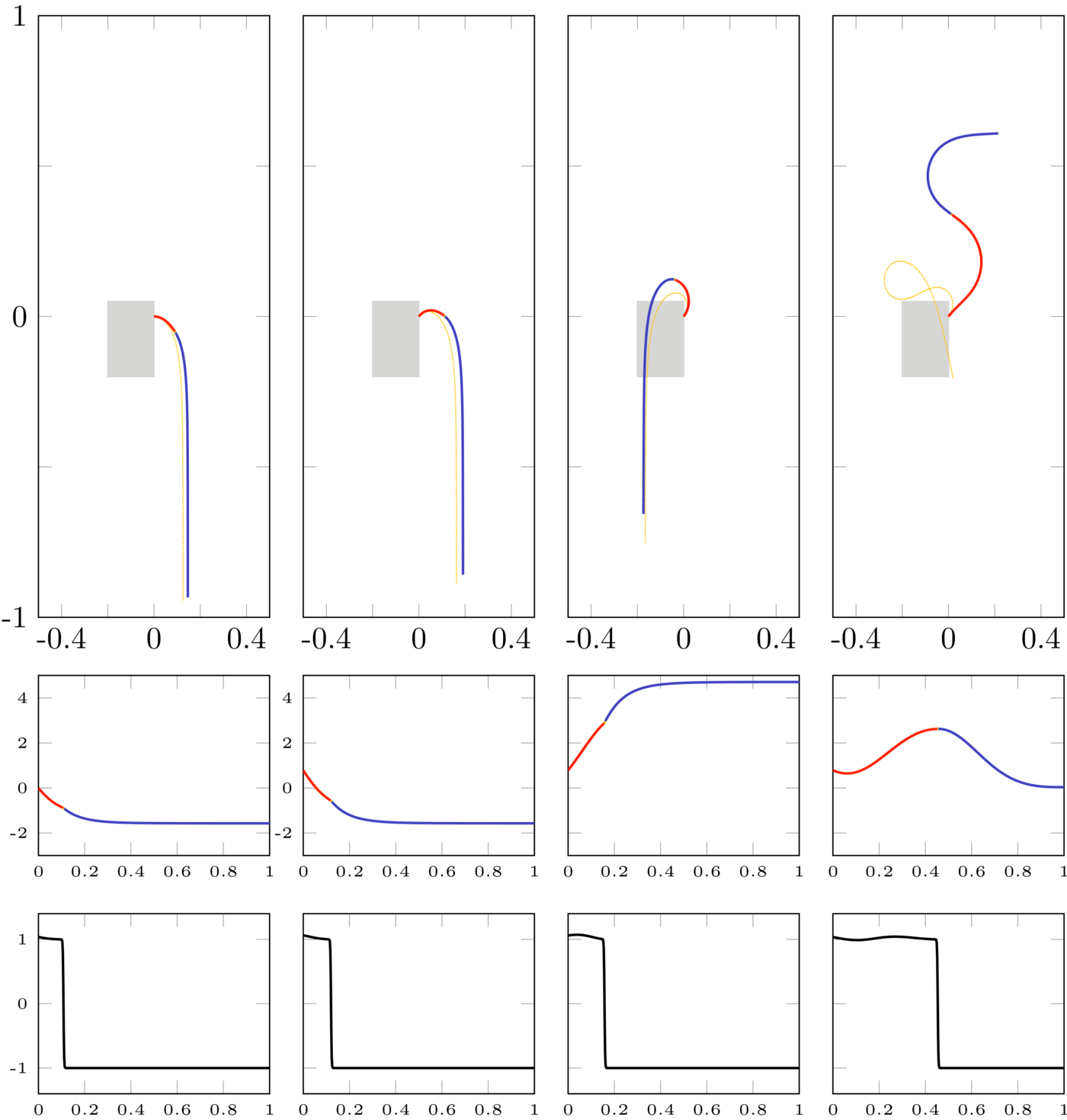}
}
\caption{Top row: Starting from different initializations (dotted orange) we obtain optimal designs (from left to right) for a simple configuration with $K_0=0$ and with $K_0=\tfrac{\pi}{4}$, as well as a twisted configuration with $K_0=\tfrac{\pi}{4}$, and an S-configuration with $K_0=\tfrac{\pi}{4}$. 
In the middle and bottom row we see the corresponding plots of the phase $K$ and phase field $v$.
Here, we have chosen $b=1$, $a=0.5$, $\delta=100$, $c_l = 1$, $c_p =1$, $N=513$, and $\epsilon=\tfrac{1}{N-1}$.
The phase field $v$ is color coded (top and middle) as}
\label{fig:optimalDesign}
\begin{center}
\resizebox{0.4\linewidth}{!}{
\begin{tikzpicture}
   \begin{axis}
     [ hide axis,
       width=1.0\linewidth,
       height=0.2\linewidth,
       at={(0,0)},
       colorbar horizontal,
       point meta min = -1.2,
       point meta max = 1.2,
     ]
  \end{axis}
\end{tikzpicture}
}
\end{center}
\end{figure}
Finally, we have implemented additional constraints prescribing a set of beam positions on $(0,1]$.
In this case, the resulting optimal designs is characterized by separated subintervales with hard material.
Also in these tests we never observed the microstructures even for small values of $c_p$.
Figure~\ref{fig:optimalDesignFixedNodes} shows an instance of these computational results with additional point constraints.
\begin{figure}[!htbp]
\resizebox{\linewidth}{!}{
 \includegraphics{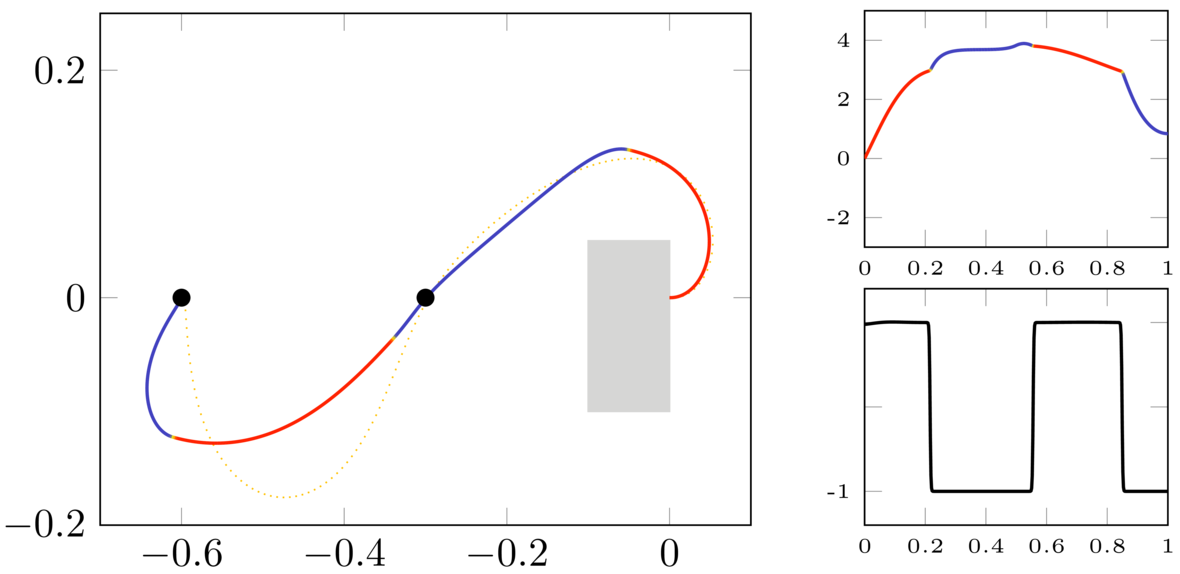}
}
\caption{Optimal designs for a beam under the constraint that three fixed beam positions $(0,0)$, $(-0.3,0)$ and $(-0.6,0)$ at times $t=0,\, 0.5,\, 1$.
Here $b=4.0$, $a=0.5$, $\delta=100$, $c_l = 1$, $c_p = 1$, $N=513$, and $\epsilon=\tfrac{1}{N-1}$. 
Plots and color coding are as in Fig.~\ref{fig:optimalDesign}.}
\label{fig:optimalDesignFixedNodes}
\end{figure}

\paragraph*{Acknowledgements.} We acknowledge support 
by the German Science Foundation via the CRC 1060 and grant no. HO 4697/1-1.
It is a pleasure to thank Matth\"aus Pawelczyk for helpful comments.

\end{document}